\documentclass{amsart}[12pt]
\usepackage[utf8]{inputenc}
\usepackage{amsmath}
\usepackage{amsfonts}
\usepackage{amssymb}

\newtheorem{theorem}{Theorem}
\newtheorem{lemma}{Lemma}
\newtheorem{proposition}{Proposition}

\theoremstyle{definition}

\newtheorem{remark}{Remark}

\begin{document}
\title[Volterra type operator]%
{Strict singularity of Volterra type operators on Hardy spaces}

\author{Qingze Lin, Junming Liu*, Yutian Wu}
\thanks{*Corresponding author}

\address{School of Applied Mathematics, Guangdong University of Technology, Guangzhou, Guangdong, 510520, P.~R.~China}\email{gdlqz@e.gzhu.edu.cn}

\address{School of Applied Mathematics, Guangdong University of Technology, Guangzhou, Guangdong, 510520, P.~R.~China}\email{jmliu@gdut.edu.cn}

\address{School of Financial Mathematics \& Statistics, Guangdong University of Finance, Guangzhou, Guangdong, 510521, P.~R.~China}\email{26-080@gduf.edu.cn}

\begin{abstract}
In this paper, we first characterize the boundedness and compactness of Volterra type operator $S_gf(z) = \int_0^z f'(\zeta)g(\zeta)d\zeta, \ z \in \mathbb{D},$ defined on Hardy spaces $H^p, \, 0< p <\infty$. The spectrum of $S_g$ is also obtained. Then we prove that $S_g$ fixes an isomorphic copy of $\ell^p$ and an isomorphic copy of $\ell^2$ if the operator $S_g$ is not compact on $H^p (1\leq p<\infty)$. In particular, this implies that the strict singularity of the operator $S_g$ coincides with the compactness of the operator $S_g$ on $H^p$. At last, we post an open question for further study.
\end{abstract}
\thanks{This work was supported by NNSF of China (Grant No. 11801094).}

\maketitle

\section{\bf Introduction}
Let $\mathbb{D}$ be the unit disk of the complex plane $\mathbb{C}$ and $H(\mathbb D)$ the space consisting of all analytic functions on $\mathbb{D}$. Then for $0<p<\infty$, the Hardy space $H^{p}$ on $\mathbb{D}$ consists of all analytic functions $f\in H(\mathbb D)$ satisfying
$$H^p:=\{f\in H(\mathbb{D} ):\ \|f\|_{H^{p}}=\left(\lim_{ r\to1^-}\int_{\partial\mathbb{D}}|f(r\xi)|^{p}dm(\xi)\right)^{1/p}<\infty\}\,,$$
where $m$ is the normalized Lebesgue measure on $\partial\mathbb{D}$. By \cite[Theorem~2.6]{DUREN}, this norm is equal to the following norm:
$$H^p:=\{f\in H(\mathbb{D} ):\ \|f\|_{H^{p}}=\left(\int_{\partial\mathbb{D}}|f(\xi)|^{p}dm(\xi)\right)^{1/p}<\infty\}\,,$$
where for any $\xi\in\partial\mathbb{D}$, $f(\xi)$ is the radial limit which exists almost everywhere (see \cite[Theorem~9.4]{zhu}).

When $p=\infty$, the space $H^{\infty}$ is defined by
$$H^{\infty}=\{f\in H(\mathbb D):\ \|f\|_{\infty}:=\ \sup_{z\in \mathbb D}\{|f(z)|\}<\infty\}\,.$$

For any analytic function $g\in H(\mathbb D)$, there are two kinds of Volterra type operators defined, respectively, by
$$(T_gf)(z)=\int_0^z f(\omega)g'(\omega)d\omega,\quad z\in\mathbb D, f\in H(\mathbb D)\,,$$
and
$$(S_gf)(z)=\int_0^z f'(\omega)g(\omega)d\omega,\quad z\in\mathbb D, f\in H(\mathbb D)\,.$$

The boundedness and compactness of these two operators on some spaces of analytic functions were extensively studied. Pommerenke \cite{P} firstly studied the boundedness of $T_g$ on Hardy-Hilbert space $H^2$. After his work, Aleman, Siskakis and Cima \cite{AC,AS} systematically studied the boundedness and compactness of $T_g$ on Hardy space $H^p$, in which they showed that $T_g$ is bounded (or compact) on $H^p,\ 0<p<\infty$, if and only if $g\in BMOA$ (or $g\in VMOA$). What's more, Aleman and Siskakis \cite{AS1} studied the boundedness and compactness of $T_g$ on Bergman spaces while Galanopoulos, Girela and Pel\'{a}ez \cite{GGP,GP} investigated the boundedness of $T_g$ and $S_g$ on Dirichlet type spaces and Xiao \cite{XJ} studied $T_g$ and $S_g$ on $Q_p$ spaces.

Recently, Lin, et al \cite{LIN} studied the boundedness of $T_g$ and $S_g$ acting on the derivative Hardy spaces $S^p$. For these operators on other spaces like Fock spaces and weighted Banach spaces, see \cite{AJS,OC,CPPR,LIN1,TM1,TM2,SSV} and the references therein.

A bounded operator $T\colon X \to Y$ between Banach spaces is strictly singular if its restriction to any infinite-dimensional closed subspace is not an isomorphism onto its image. This notion was introduced by Kato \cite{Kato}. The obvious example of strictly singular non-compact operators are inclusion mappings $i_{p,q} \colon \ell^p \hookrightarrow \ell^q,$ when $1 \le p < q <\infty.$

A bounded operator $T\colon X \to Y$ between Banach spaces is said to fix a copy of the given Banach space $E$ if there is a closed subspace $M\subset X$, linearly isomorphic to $E$, such that the restriction $T_{|M}$ defines an isomorphism from $M$ onto $T(M)$. The bounded operator $T\colon X \to Y$ is called $\ell^p$-singular if it does not fix any copy of $\ell^p$\,.

Miihkinen \cite{SM} studied the strict singularity of $T_g$ on Hardy space $H^p$ and showed that the strict singularity of $T_g$ coincides with its compactness on $H^p, \, 1 \leq p < \infty.$ whose main ideas come from the recent paper \cite{LNST} where the corresponding questions are investigated for composition operators.

Although the boundedness and compactness of the operator $T_g$ on $H^p$ had been studied, from the literature that we have looked at so far, the proofs of the boundedness and compactness for the operator $S_g$ on $H^p$ are still not been shown in detail, except for the case $p=2$ whose study seems to be elementary (see \cite{LS}). Thus, in this paper, We first characterize the boundedness and compactness of Volterra type operator $S_gf$ defined on Hardy spaces $H^p$ for $0 < p < \infty$\,. Base on the characterization of the boundedness for the operator $S_g$ on $H^p$, we are able to characterize the spectrum of $S_g$ on $H^p$, inspired by the idea in the papers \cite{OC,OCAP}. Then we prove that the bounded operator $S_g$ fixes an isomorphic copy of $\ell^p$ if the operator $S_g$ is not compact on $H^p$. In particular, this implies that the strict singularity of the operator $S_g$ coincides with the compactness of the operator $S_g$ on $H^p.$ Moreover, we show that $S_g$, when acting on $H^p (1\leq p<\infty)$, fixes an isomorphic copy of $\ell^2$.

In the last section, we post an open question for further study.

Our main results are as follows:

\begin{proposition}
\label{pro1}
Let $g \in H(\mathbb{D})$ and $0 < p <\infty.$ Then the operator $S_g \colon H^p \to H^p$ is bounded if and only if $g\in H^\infty$.
\end{proposition}

\begin{proposition}
\label{pro2}
Let $g \in H(\mathbb{D})$ and $0 < p <\infty.$ Then the operator $S_g \colon H^p \to H^p$ is compact if and only if $g=0$.
\end{proposition}

\begin{proposition}
\label{pro3}
Let $g \in H(\mathbb{D})$ and $0 < p <\infty.$ Then the spectrum of the bounded operator $S_g \colon H^p \to H^p$ is $\sigma(S_g)=\{0\}\cup\overline{g(\mathbb{D})}$.
\end{proposition}

\begin{theorem}
\label{th1}
Let $1 \leq p <\infty$ and suppose that $S_g \colon H^p \to H^p$ is bounded but not compact. Then the operator $S_g \colon H^p \to H^p$ fixes an isomorphic copy of $\ell^p$ and an isomorphic copy of $\ell^2$ as well. In particular, the operator $S_g$ is not strictly singular, that is, strict singularity of bounded operator $S_g$ coincides with its compactness.
\end{theorem}

\section{Boundedness and Compactness of $S_g$ on $H^p$}
In this section, we provide a new proof for the conditions of boundedness and compactness of the operator $S_g$ on $H^p$ when $0<p<\infty$. Although the result for the boundedness can be deduced from \cite[Lemma~2.1(i)]{AJS}, we give our new proof which is not only useful to the proof of compactness of $S_g$ on $H^p$ and the proof of Theorem~\ref{th1}, but also of interest itself.

\begin{proof}[Proof of Proposition~\ref{pro1}]
Assume that $S_g \colon H^p \to H^p$ is bounded. From \cite{ZHAO}, we know that $H^p\subset BMOA^{1+1/p}_p$, where $BMOA^{1+1/p}_p$ is the space of analytic functions $f$ satisfying
$$\sup_{a\in\mathbb{D}}\int_{\mathbb{D}}|f'(z)|^p(1-|z|^2)^{p-1}(1-|\varphi_a(z)|^2)dA(z)<\infty$$
in which $\varphi_a(z)=\frac{a-z}{1-\bar{a}z}$ is the M\"{o}bius transformation on $\mathbb{D}$ and $A$ is the normalized Lebesgue measure on $\mathbb{D}$.
Hence, if $S_g \colon H^p \to H^p$ is bounded, then $S_g \colon H^p \to BMOA^{1+1/p}_p$ is also bounded. Therefore, for any $f\in H^p$, we have
$$\sup_{a\in\mathbb{D}}\int_{\mathbb{D}}|(S_gf)'(z)|^p(1-|z|^2)^{p-1}(1-|\varphi_a(z)|^2)dA(z)\leq C\|f\|_{H^p}^p\,.$$
It is easy to verify that for any $a\in\mathbb{D}$, the function
$$f_a(z)=\frac{(1-|a|^2)^{2-1/p}}{(1-\bar{a}z)^2}$$
is a unit vector in $H^p$. Thus,
$$\sup_{a\in\mathbb{D}}\int_{\mathbb{D}}|(S_gf_a)'(z)|^p(1-|z|^2)^{p-1}(1-|\varphi_a(z)|^2)dA(z)\leq C\|f_a\|_{H^p}^p\,,$$
or equivalently,
$$\sup_{a\in\mathbb{D}}\int_{\mathbb{D}}|\bar{a}|^p|g(z)|^p\frac{(1-|a|^2)^{2p-1}}{|1-\bar{a}z|^{3p}}(1-|z|^2)^{p-1}(1-|\varphi_a(z)|^2)dA(z)\leq C\,.$$
Then, let $z=\varphi_a(\omega)=\frac{a-\omega}{1-\bar{a}\omega}$ be a M\"{o}bius transformation on $\mathbb{D}$, we have
$$\sup_{a\in\mathbb{D}}\int_{\mathbb{D}}|\bar{a}|^p|g(\varphi_a(\omega))|^p\frac{(1-|a|^2)^{2p-1}}{|1-\bar{a}\varphi_a(\omega)|^{3p}}(1-|\varphi_a(\omega)|^2)^{p-1}(1-|\omega|^2)|\varphi'_a(\omega)|^2dA(z)\leq C\,.$$
Note that $|1-\bar{a}\varphi_a(\omega)|=(1-|a|^2)/|1-\bar{a}\omega|$ and $(1-|\varphi_a(\omega)|^2)=(1-|z|^2)||\varphi'_a(\omega)|$, we obtain
$$\sup_{a\in\mathbb{D}}\int_{\mathbb{D}}|\bar{a}|^p|g(\varphi_a(\omega))|^p|1-\bar{a}\omega|^{p-2}(1-|\omega|^2)^pdA(z)\leq C\,.$$

Now, consider the analytic function $G_a(\omega):=\bar{a}^pg(\varphi_a(\omega))^p(1-\bar{a}\omega)^{p-2}$, we get that
$$\sup_{a\in\mathbb{D}}|G_a(0)|=\sup_{a\in\mathbb{D}}|\bar{a}|^p|g(a)|^p\leq\sup_{a\in\mathbb{D}}\int_{\mathbb{D}}|\bar{a}|^p|g(\varphi_a(\omega))|^p|1-\bar{a}\omega|^{p-2}(1-|\omega|^2)^pdA(z)\leq C$$
which implies that
$$\sup_{|a|\rightarrow1^{-}}|g(a)|^p\leq C,$$ that is, $g\in H^\infty$.

Conversely, assume that $g\in H^\infty$, then by \cite{AS,ZCRZ}, the operator $T_g \colon H^p \to H^p$ and the multiplication operator $M_g \colon H^p \to H^p$ are both bounded. Therefore, it follows from the obvious equality $(M_gf)(z)-(M_gf)(0)=(T_gf)(z)+(S_gf)(z)$ that $S_g \colon H^p \to H^p$ is also bounded. Accordingly, the proof is complete.
\end{proof}

\begin{remark}
We note that the sufficiency of Proposition~\ref{pro1} can also be proven directly by using the following equivalent norms for $H^p$ (see \cite[p.~125]{AB}):
$$\|f\|^p_{H^p}\asymp\int_{\partial\mathbb{D}}\left(\int_{S(\xi)}|f^{(n)}(z)|^2(1-|z|^2)^{2n-2}dA(z)\right)^{p/2}dm(\xi)+\sum_{j=0}^{n-1}|f^{(j)}(0)|^p\,.$$
\end{remark}

\begin{proof}[Proof of Proposition~\ref{pro2}]
It is obvious that if $g=0$, then $S_g \colon H^p \to H^p$ is compact.

Conversely, if $S_g \colon H^p \to H^p$ is compact, $S_g \colon H^p \to BMOA^{1+1/p}_p$ is also compact. Since for any sequence $\{a_n\}_{n=1}^\infty$ such that $\lim_{n\rightarrow\infty}|a_n|=1$, $f_{a_n}$ converges to $0$ uniformly on compact subsets of $\mathbb{D}$, it holds that
$$\lim_{n\rightarrow\infty}\int_{\mathbb{D}}|(S_gf_{a_n})'(z)|^p(1-|z|^2)^{p-1}(1-|\varphi_{a_n}(z)|^2)dA(z)=0\,,$$
then similar to the arguments in the proof of Proposition~\ref{pro1}, we obtain
$$\lim_{n\rightarrow\infty}|g(a_n)|^p=0\,.$$
That is, $g=0$. Accordingly, the proof is complete.
\end{proof}

\section{The spectrum of $S_g$ on $H^p$}
In this section, we characterize the the spectrum of the bounded operator $S_g$ on $H^p$\,.

\begin{proof}[Proof of Proposition~\ref{pro3}]
Since for any $f\in S^p\setminus\{0\}$, the function $S_gf$ has a zero at $z=0$, it holds that $0\in\sigma(S_g)$.

Now, we assume that $\lambda\in \mathbb{C}\setminus\{0\}$. For any $h\in H(\mathbb{D})$, it is easy to show that the equation
$$f-\frac{1}{\lambda}S_gf=h$$
has the unique solution $f$ in $H(\mathbb{D})$ and the solution is
$$f(z)=R_{\lambda,g}h(z):=\int_{0}^z\frac{h'(\zeta)}{1-\frac{1}{\lambda}g(\zeta)}d\zeta+h(0)\,.$$
Therefore, the resolvent set $\rho(S_g)$ of the bounded operator $S_g$ consists precisely of all points $\lambda\in \mathbb{C}$ for which $R_{\lambda,g}$ is a bounded operator on $H^p$\,.

If $\lambda\in\mathbb{C}\setminus(\{0\}\cup\overline{g(\mathbb{D})})$, then $1-\frac{1}{\lambda}g(\zeta)$ is bounded away from $0$, that is, $\frac{1}{1-\frac{1}{\lambda}g(\zeta)}$ is bounded. Thus,
$$f=S_{(1-\frac{1}{\lambda}g)}h+h(0)\in H^p$$
by Proposition~\ref{pro1}, which implies that the operator $R_{\lambda,g}$ is a bounded operator on $H^p$\,. Accordingly, $\mathbb{C}\setminus(\{0\}\cup\overline{g(\mathbb{D})})\subset\rho(S_g)$, that is, $\sigma(S_g)\subset(\{0\}\cup\overline{g(\mathbb{D})})$\,.

Conversely, if $\lambda\in g(\mathbb{D})$ and $\lambda\neq0$, then $\frac{1}{1-\frac{1}{\lambda}g(\zeta)}$ is not bounded, which implies that the operator $R_{\lambda,g}$ is not bounded on $H^p$, hence we have $g(\mathbb{D})\setminus\{0\}\subset\sigma(S_g)$\,. Thus, in conjunction with the fact that $0\in\sigma(S_g)$, it holds that
$$g(\mathbb{D})\cup\{0\}\subset\sigma(S_g)\subset\overline{g(\mathbb{D})}\cup\{0\}\,.$$
Since the spectrum $\sigma(S_g)$ is closed, we obtain that $\sigma(S_g)=\overline{g(\mathbb{D})}\cup\{0\}$\,.
\end{proof}

\section{Proof of Theorem~\ref{th1}}
First, we note that Theorem~\ref{th1} holds for $p=2$ due to the fact that a bounded linear operator on $H^2$ is compact if and only if it is strict singular, if and only if it does not fix any copy of $\ell^2$ (see \cite[5.1-5.2]{AP})\,. From now on, we suppose that $p\neq 2$\,.

From the proof in Proposition~\ref{pro2}, it can be easily checked that, if the bounded operator $S_g \colon H^p \to H^p$ is not compact, then there exists a sequence $(a_n) \subset \mathbb{D}$ with $0 < |a_1| < |a_2| < \ldots < 1$ and $a_n \to \omega \in \partial\mathbb{D},$ such that there is a positive constant $h$ such that
$$\|S_g(f_{a_n})\|_{H^p}\geq h>0$$
holds for all $n\in\mathbb{N}$ and $f_{a_n}$ defined in the previous section. We may assume without loss of generality that $a_n \rightarrow1$ as $n\rightarrow\infty$ by utilizing a suitable rotation.

\begin{lemma}
\label{le1}
Let $(a_n) \subset \mathbb{D}$ be a sequence as above. Let $A_\varepsilon = \{e^{i\theta}: |e^{i\theta}-1| < \varepsilon\}$ for each $\varepsilon > 0$. Then for bounded operator $S_g \colon H^p \to H^p$, we have
\begin{align*}
&\textrm{(1) $\lim_{\varepsilon \to 0}\int_{A_\varepsilon}|S_g f_{a_n}|^p dm = 0$\quad for every $n\in\mathbb{N}$.}&\\
&\textrm{(2) $\lim_{n \rightarrow \infty}\int_{\partial\mathbb{D}\setminus A_\varepsilon}|S_g f_{a_n}|^p dm = 0$\quad for every $\varepsilon > 0.$}&
\end{align*}
\end{lemma}
\begin{proof}
\textbf{(1)} For each fixed $n$, this follows immediately from the absolute continuity of Lebesgue measure and the boundedness of operator $S_g \colon H^p \to H^p$\,.

\textbf{(2)} For given $\varepsilon > 0$, it is easy to see that there is a positive $\gamma>0$ such that $|1-\bar{a}_nre^{i\vartheta}|\geq \gamma$ for all $n\in \mathbb{N}$, $0\leq r<1$ and $\varepsilon\leq \vartheta\leq\pi$\,. Therefore, for these $r$ and $\vartheta$, we get that
$$|f'_{a_n}(re^{i\vartheta})|^p=\frac{|\bar{a}_n|^p(1-|a|^2)^{2p-1}}{|1-\bar{a}_nre^{i\vartheta}|^{3p}}\leq\frac{|\bar{a}_n|^p(1-|a_n|^2)^{2p-1}}{\gamma^{3p}}\,,$$
for all $n\in \mathbb{N}$.
Then, for any $\xi\in\partial\mathbb{D}\setminus A_\varepsilon$, we have
\begin{equation}\begin{split}\nonumber
|(S_gf_{a_n})(\xi)|^p&=\left|\int_0^1f'_{a_n}(r\xi)g(r\xi)\xi dr\right|^p
\leq \left(\int_0^1|f'_{a_n}(r\xi)g(r\xi)|dr\right)^p\\
&\leq \|g\|^p_{\infty}\left(\int_0^1|f'_{a_n}(r\xi)|dr\right)^p
\leq \|g\|^p_{\infty}\frac{|\bar{a}_n|^p(1-|a_n|^2)^{2p-1}}{\gamma^{3p}}\,.
\end{split}\end{equation}
Accordingly,
$$\lim_{n \rightarrow \infty}\int_{\partial\mathbb{D}\setminus A_\varepsilon}|S_g f_{a_n}|^p dm\leq \lim_{n \rightarrow \infty}\|g\|^p_{\infty}\frac{|\bar{a}_n|^p(1-|a_n|^2)^{2p-1}}{\gamma^{3p}}= 0\,.$$
The proof is complete.
\end{proof}

Now, we are prepared to give a proof of Theorem~\ref{th1}.
\begin{proof}[Proof of Theorem~\ref{th1}]
First, as noted above, there exists a sequence $(a_n) \subset \mathbb{D}$ with $0 < |a_1| < |a_2| < \ldots < 1$ and $a_n \to 1,$ such that there is a positive constant $h$ such that
$\|S_g(f_{a_n})\|_{H^p}\geq h>0$ holds for all $n\in\mathbb{N}$\,.

Then by Lemma~\ref{le1} and induction method, we can find a decreasing  positive sequence $(\varepsilon_n)$ such that $A_{\varepsilon_1}=\partial\mathbb{D}$ and $\lim_{n\rightarrow\infty}\varepsilon_n=0$, and a subsequence $(b_n) \subset (a_n)$ such that the following three conditions hold:
\begin{eqnarray*}
&\textup{(1)}& \left(\int_{A_n} |S_g f_{b_k}|^p dm \right)^{1/p} < 4^{-n} \delta h, \quad k = 1,\ldots, n - 1; \\
&\textup{(2)}& \left(\int_{\partial\mathbb{D} \setminus A_n} |S_g f_{b_n}|^p dm \right)^{1/p} < 4^{-n} \delta h; \\
&\textup{(3)}& \left(\int_{A_n} |S_g f_{b_n}|^p dm \right)^{1/p} > \frac{h}{2}
\end{eqnarray*}
for every $n \in \mathbb{N},$ where $A_n = A_{\varepsilon_n}$ and $\delta > 0$ is a small constant whose value will be determined later.

Now we are ready to prove that $\|\sum_{j=1}^\infty c_jS_g(f_{b_j})\|_{H^p} \geq C \|(c_j)\|_{\ell^p},$ where the constant $C > 0$ may depend on $p.$

\begin{equation}\begin{split}\nonumber
&\phantom{=}\|\sum_{j=1}^\infty c_jS_g(f_{b_j})\|_{H^p}^p=\sum_{n = 1}^\infty \int_{A_n \setminus A_{n+1}}\left|\sum_{j=1}^\infty c_jS_g(f_{b_j})\right|^p dm\\
&\geq\sum_{n = 1}^\infty \left(|c_n| \left(\int_{A_n \setminus A_{n+1}}| S_g f_{b_n}|^p dm\right)^{1/p}
-\sum_{j \neq n}|c_j|\left(\int_{A_n \setminus A_{n+1}}| S_g f_{b_j}|^p dm\right)^{1/p} \right)^p\,.
\end{split}\end{equation}

Observe that for every $n \in \mathbb{N},$ we have
\begin{equation}\begin{split}\nonumber
\left(\int_{A_n \setminus A_{n+1}}| S_g f_{b_n}|^p dm\right)^{1/p}&=\left(\int_{A_n}| S_g f_{b_n}|^p dm-\int_{A_{n+1}}| S_g f_{b_n}|^p dm\right)^{1/p}\\
&\geq \left(\left(\frac{h}{2}\right)^p-\left(4^{-n-1} \delta h\right)^p\right)^{1/p}\geq \frac{h}{2}-4^{-n-1} \delta h
\end{split}\end{equation}
according to conditions (1) and (3) above, where the last estimate holds for $1\leq p<\infty$.

Moreover, we have
$$\left( \int_{A_n \setminus A_{n+1}}|S_g(f_{b_j})|^p dm \right)^{1/p} \leq \left( \int_{A_n }|S_g(f_{b_j})|^p dm \right)^{1/p} < 4^{-n}\delta h$$ for $j < n$ by condition (1) and
$$\left( \int_{A_n \setminus A_{n+1}}|S_g f_{b_j}|^p dm \right)^{1/p} \le \left( \int_{\partial\mathbb{D} \setminus A_j }|S_g f_{b_j}|^p dm \right)^{1/p} < 4^{-j}\delta h$$ for $j > n$ by condition (2).

Thus it always holds that $$\left(\int_{A_n \setminus A_{n+1}}|S_g f_{b_j}|^p dm \right)^{1/p} < 2^{-n-j}\delta h\quad\text{ for } j \ne n.$$

Consequently, by the triangle inequality in $L^p$, we obtain that
\begin{equation}\begin{split}\nonumber
\|\sum_{j=1}^\infty c_jS_g(f_{b_j})\|_{H^p}&\geq\left(\sum_{n = 1}^\infty \left(|c_n| \left(\frac{h}{2}-4^{-n-1} \delta h\right)
-2^{-n}\delta h\|(c_j)\|_{\ell^p} \right)^p\right)^{1/p}\\
&\geq \left(\sum_{n = 1}^\infty \left(|c_n| \left(\frac{h}{2}\right)
-2^{-n+1}\delta h\|(c_j)\|_{\ell^p} \right)^p\right)^{1/p}\\
&\geq \frac{h}{2}\|(c_j)\|_{\ell^p}-\delta h\|(c_j)\|_{\ell^p} \left(\sum_{n = 1}^\infty2^{-(n-1)p}\right)^{1/p}\\
&\geq h\left(\frac{1}{2}-\delta\left(1-2^{-p}\right)^{-1/p}\right)\|(c_j)\|_{\ell^p}\geq C\|(c_j)\|_{\ell^p}\,,
\end{split}\end{equation}
where the last inequality holds when we choose $\delta$ small enough.

A straightforward variant of the above procedure also gives $$\|\sum_{j=1}^\infty c_jS_g(f_{b_j})\|_{H^p} \leq C_1 \|(c_j)\|_{\ell^p},$$ where the constant $C_1 > 0$ may depend on $p.$

By choosing $g=1$ and the fact that $\lim_{n\rightarrow\infty}f_{a_n}(0)=0$, we obtain that
$$C_2 \|(c_j)\|_{\ell^p}\leq \|\sum_{j=1}^\infty c_jf_{b_j}\|_{H^p} \leq C_3 \|(c_j)\|_{\ell^p}\,.$$

Thus, we have
$$\|\sum_{j=1}^\infty c_jS_g(f_{b_j})\|_{H^p} \geq C \|(c_j)\|_{\ell^p} \geq CC_3^{-1}\|\sum_{j=1}^\infty c_jf_{b_j}\|_{H^p}\,$$
This implies that the operator $S_g \colon H^p \to H^p$ fixes an isomorphic copy of $\ell^p$\,.

To prove that the operator $S_g \colon H^p \to H^p$ fixes an isomorphic copy of $\ell^2$, we consider the trivial equality
$$M_g(f)=f(0)g(0)+T_g(f)+S_g(f)\quad \text{ for } z\in \mathbb D\,, f\in H(\mathbb D),$$
where $M_g$ is the multiplication operator. It is proven that $M_g$ is not $l^2$-singular on $H^p$ while $T_g$ is $l^2$-singular on $H^p$ (see \cite{LMN,MNST}). Since the class of $l^2$-singular operators forms a linear subspace of the space of all bounded operators, it follows from the above equality that $S_g$ is not $l^2$-singular on $H^p$, that is, $S_g \colon H^p \to H^p$ fixes an isomorphic copy of $\ell^2$.
The proof is complete.
\end{proof}

\section{Open question}
In Theorem~\ref{th1}, we exclude the case of $p=\infty$ since our method is not applicable in this case. Although Bourgain \cite{JB} has established that a bounded linear operator on $H^\infty$ is weakly compact if and only if it does not fix any copy of $\ell^\infty$, we do not know whether or not $S_g$ acting on $H^\infty$ is compact if and only if it does not fix any copy of $\ell^\infty$. So we post it as an open question as follows:

{\it {\bf Open Question.} Suppose that $S_g \colon H^\infty \to H^\infty$ is bounded but not compact. Does the operator $S_g \colon H^\infty \to H^\infty$ fix an isomorphic copy of $\ell^\infty$? In particular, is the operator $S_g$ not strictly singular?

\end{document}